\documentclass{article}

 \PassOptionsToPackage{numbers, compress}{natbib}

 \usepackage[preprint]{nips_2018}

\usepackage[utf8]{inputenc} 
\usepackage[T1]{fontenc}    
\usepackage{hyperref}       
\usepackage{url}            
\usepackage{booktabs}       
\usepackage{amsfonts}       
\usepackage{nicefrac}       
\usepackage{microtype}      

\usepackage[utf8]{inputenc} 
\usepackage[english]{babel} 

\usepackage{graphicx}
\usepackage{mathtools, cuted}
\usepackage{hyperref}
\usepackage{amsmath}
\usepackage{amssymb}
\usepackage{mathtools}
\usepackage{amsthm}

\newtheorem{theorem}{Theorem}[section]
\newtheorem{definition}[theorem]{Definition}
\newtheorem{corollary}[theorem]{Corollary}
\newtheorem{lemma}[theorem]{Lemma} 
 
\newtheorem{remark}[theorem]{Remark}

\title{Fundamentals of Gaussian CM Sequences}

\author{
  Reza Rezaie and X. Rong Li \\
 Department of Electrical Engineering\\
 University of New Orleans\\
New Orleans, LA 70148 \\
  \texttt{rrezaie@uno.edu} and \texttt{xli@uno.edu} \\
}

\begin{document}

\maketitle

\begin{abstract}
Markov processes are widely used in modeling random phenomena/problems. However, they may not be adequate in some cases where more general processes are needed. The conditionally Markov (CM) process is a generalization of the Markov process based on conditioning. There are several classes of CM processes (one of them is the class of reciprocal processes), which provide more capability (than Markov) for modeling random phenomena. Reciprocal processes have been used in many different applications (e.g., image processing, intent inference, intelligent systems). In this paper, nonsingular Gaussian (NG) CM sequences are studied, characterized, and their dynamic models are presented. The presented results provide effective tools for studying reciprocal sequences from the CM viewpoint, which is different from that of the literature. Also, the presented models and characterizations serve as a basis for application of CM sequences, e.g., in motion trajectory modeling with destination information.
\end{abstract}

\textbf{Keywords:} Conditionally Markov (CM) sequence, Gaussian sequence, dynamic model, characterization.

\section{Introduction}

For modeling a random phenomenon/problem, usually the following order should be considered \cite{Li_2015}. First, if the phenomenon is time-invariant, a random variable might be good enough. Otherwise, a stochastic process seems necessary. An independent process can be considered first for simplicity. If such a simple process is not good enough, the next choice is usually a Markov process. The Markov process has two elements (i.e., an initial density and an evolution law). Even the Markov process is not good enough for some cases. Sometimes a higher order (e.g., second order) Markov process does not fit some phenomena well, for example, a time-varying phenomenon with some information available about its future (e.g., destination). More specifically, consider an example of trajectory prediction with destination information \cite{Fanas1}--\cite{DD_Conf}. Such a problem has three main elements: an origin, an evolution law, and a destination, for which the Markov process does not fit since it can not model information about the destination. In other words, the destination density of a Markov process is completely determined by its initial density and evolution law. One class of CM processes called $CM_L$ has the following main elements: a joint endpoint density and an evolution law (in other words, an initial density, an evolution law, and a destination density conditioned on the initial). This process can model destination information while it has a Markov-like evolution law, which is desirable for simplicity. Therefore, the $CM_L$ process is more suitable than the Markov process for problems with information about the destination. Generally speaking, CM processes, including the Markov process as a special case, provide a systematic approach and a wide variety of choices for modeling random phenomena. Conditioning is a very powerful concept/tool in probability theory. The notion of CM processes combines the conditioning concept and the Markov property. Different ways of combining the two lead to different classes of CM processes, which are more powerful than the Markov process for modeling random phenomena. In addition, as a special CM process, the reciprocal process has been used in many different areas of science and engineering, including stochastic mechanics, image processing, trajectory modeling and intent inference, intelligent systems, and acausal systems (e.g., \cite{Levy_1}--\cite{Krener1}). CM processes provide a fruitful viewpoint for studying the reciprocal process from the CM viewpoint \cite{CM_Part_II_A_Conf}. Therefore, it is desired to study, model, and characterize different classes of CM processes. 

Consider stochastic sequences defined over $[0,N]=\lbrace 0,1,\ldots,N \rbrace$. For convenience, let the index be time. A sequence is Markov if and only if (iff) conditioned on the state at any time $k$, the subsequences before and after $k$ are independent. A sequence is reciprocal iff conditioned on the states at any two times $k_1$ and $k_2$, the subsequences inside and outside the interval $[k_1,k_2]$ are independent. In other words, inside and outside are independent given the boundaries. A sequence is $CM_F$ ($CM_L$) iff conditioned on the state at time $0$ ($N$), the sequence is Markov over $[1,N]$ ($[0,N-1]$). The subscript ``$F$" (``$L$") is used because the conditioning is at the \textit{first} (\textit{last}) time of the interval. There are other classes that are CM over a subinterval $[k_1,k_2] \subset [0,N]$. But in this paper we do not consider them. The Markov sequence and the reciprocal sequence are two important classes of the CM sequence.

The notion of CM processes was introduced in \cite{Mehr} for Gaussian processes based on mean and covariance functions. Stationary Gaussian CM processes were studied and characterized, and construction of some non-stationary Gaussian CM processes was discussed. \cite{ABRAHAM} extended the definition of Gaussian CM processes (presented in \cite{Mehr}) to the general (Gaussian/non-Gaussian) case. Furthermore, \cite{ABRAHAM} and \cite{Carm} commented on the relation of Gaussian CM \cite{Mehr} and Gaussian reciprocal processes. Reciprocal processes were introduced in \cite{Bernstein}, and later studied in \cite{Slepian}--\cite{Moura}. A characterization and a dynamic model of NG reciprocal sequences, and a characterization of NG Markov sequences were presented in \cite{Levy_Dynamic} and \cite{Ackner}, respectively. The relationship between the (Gaussian/non-Gaussian) CM process and the (Gaussian/non-Gaussian) reciprocal process was studied in \cite{CM_Part_II_A_Conf}, where a dynamic model governing the NG reciprocal sequence was presneted from the CM viewpoint. 

From system theory, it is well known that the state concept is equivalent to the Markov property, that is, conditioned on the state at any time, the states before and after are independent. That is why there exists a recursive model for the evolution of a Markov sequence. However, for a general sequence there is no simple recursive model for evolution. The CM sequence is more general than the Markov sequence. Consequently, a CM sequence may not have the above concept of state, in general. Instead, it has a similar concept if it is conditioned on the states at two instead of one time. That is why a simple recursive model also exists for the evolution of Gaussian CM sequences. In this paper, we start from a formal definition of CM sequences and obtain a simpler yet equivalent description of CM sequences, particularly for the Gaussian case. Then, the corresponding dynamic models and characterizations are obtained.

The main contributions of this paper are as follows. (stationary/non-stationary) NG CM sequences are studied and their dynamic models and characterizations are presented. These models and characterizations make CM sequences easily applicable. The presented dynamic models (called $CM_L$/$CM_F$ models) are recursive. We prove that every $CM_L$ ($CM_F$) sequence obeys a $CM_L$ ($CM_F$) model and every sequence governed by a $CM_L$ ($CM_F$) model is a $CM_L$ ($CM_F$) sequence. In other words, the model is a complete description of the $CM_L$ ($CM_F$) sequence (the same is true for the characterizations). This paper provides useful tools for application of NG CM sequences in different problems, e.g., motion trajectory modeling with destination information. Also, it provides a foundation for studying and modeling reciprocal sequences from the CM viewpoint, which is a very fruitful angle and leads to easily applicable results.

In this paper, in Section \ref{Section_Definition}, definitions of two main classes of CM sequences ($CM_L$ and $CM_F$) are presented for the general (Gaussian/non-Gaussian) case. In Section \ref{Section_CML_Dynamic}, dynamic models of NG $CM_L$ and NG $CM_F$ sequences are presented. Characterizations of NG $CM_L$ and NG $CM_F$ sequences are given in Section \ref{Section_CML_Characterization}. Finally, Section \ref{Section_Summary} contains conclusions and discusses applications of the obtained results.

\section{Definitions and Preliminaries}\label{Section_Definition}

\subsection{Conventions}

We consider stochastic sequences defined over the interval $[0,N]$, which is a general discrete index interval, but for convenience it is called time. Also, we define  
\begin{align*}
[i,j]& \triangleq \lbrace i,i+1,\ldots ,j-1,j \rbrace \\
[x_k]_{i}^{j} & \triangleq \lbrace x_k, k \in [i,j] \rbrace \\
[x_k] & \triangleq [x_k]_{0}^{N}\\
[x_k]_{J} & \triangleq \lbrace x_k, k \in J \rbrace, J \subset [0,N]\\
i,j,k_1,k_2& \in [0,N], i<j \\
\sigma([x_k]_i^j)& \triangleq \sigma\text{-field generated by } [x_k]_i^j
\end{align*}
$C_{k_1,k_2}$ is a covariance function, and $C_k \triangleq C_{k,k}$. $C$ is the covariance matrix of the whole sequence $[x_k]$. Also, $0$ may denote a zero scalar, vector, or matrix, as is clear from the context. The symbol "$\setminus$" is used for set subtraction. $F(\cdot | \cdot)$ denotes a conditional cumulative distribution function (CDF).

We assume the stochastic sequences are defined with respect to an underlying probability triple $(\Omega,\mathcal{A},P)$. The abbreviations ZMNG and NG are used for ``zero-mean nonsingular Gaussian" and ``nonsingular Gaussian", respectively.

\subsection{CM Definitions and Notations}\label{Section_Definition_A}

\begin{definition}\label{CML_Like}
$[x_k]$ is $CM_c, c \in \lbrace 0,N \rbrace$, if for every $j \in [0,N]$,
\begin{align}\label{CML_Like_Eq}
P \lbrace AB|x_{j},x_{c} \rbrace =P \lbrace A|x_{j},x_{c}\rbrace P \lbrace B|x_{j},x_{c}\rbrace
\end{align}
where $A \in \sigma([x_k]_{j+1}^{N}\setminus \lbrace x_c \rbrace)$ \footnote{Note: $[x_k]_{j+1}^{N}\setminus \lbrace x_{N} \rbrace=[x_k]_{j+1}^{N-1}$ and $[x_k]_{j+1}^{N} \setminus \lbrace x_{0} \rbrace=[x_k]_{j+1}^{N}$.} and $B \in $ $ \sigma([x_k]_{0}^{j-1} \setminus \lbrace x_c \rbrace)$. 

\end{definition}

In other words, a sequence $[x_k]$ is $CM_c, c \in \lbrace 0,N \rbrace$, iff conditioned on the state at time $0$ ($N$), the sequence is Markov over $[1,N]$ ($[0,N-1]$). 

To build the foundation, we need a formal definition of CM sequences (Definition \ref{CML_Like}). However, to provide the results in a simple language for application, later we present Corollary \ref{CDF} below which is equivalent to Definition \ref{CML_Like}. 

\begin{remark}\label{R_CMN}
It is convenient to consider the following notation
\begin{align*}
CM_c=\left\{ \begin{array}{cc} 
CM_F & \text{if  } c=0\\
CM_L & \text{if  } c=N
\end{array} \right.
\end{align*}

\end{remark}

There are other CM classes. But in this paper we only consider $CM_L$ and $CM_F$.

\subsection{Preliminaries (For Gaussian CM Sequences)}\label{Section_Definition_B}

We present some equations which are equivalent to the above definitions of CM sequences, particularly in the Gaussian case. Due to space limitation, we skip some proofs.

\begin{lemma}\label{CML_Definition_Expectation}
$[x_k]$ is $CM_c, c \in \lbrace 0,N \rbrace$, iff for every Borel measurable function $f$
 \begin{align}\label{CML_Expectation_1}
 E[f(x_k)|[x_{i}]_{0}^{j},x_{c}]=E[f(x_k)|x_j,x_{c}]
 \end{align} 
for every $j,k \in [0,N], j<k$.

\end{lemma}

Then, equivalent to Lemma \ref{CML_Definition_Expectation}, we have the following corollary.
\begin{corollary}\label{CDF}
$[x_k]$ is $CM_c, c \in \lbrace 0,N \rbrace$, iff
 \begin{align}
F(\xi _k|[x_{i}]_{0}^{j},x_{c})=F(\xi _k|x_j,x_c)\label{CDF_1}
 \end{align} 
for every $j,k \in [0,N], j<k$, and every $\xi _k \in \mathbb{R}^d$, where $d$ is the dimension of $x_k$, and $F(\cdot|\cdot)$ is the conditional CDF.

\end{corollary}

For the Gaussian sequence, the above results for the CM sequence are equivalent to the following.

\begin{lemma}\label{GaussianCML_Definition_Expectation}
A Gaussian $[x_k]$ is $CM_c, c \in \lbrace 0,N \rbrace$, iff
\begin{align}\label{GaussianCML_E_1}
 E[x_k|[x_{i}]_{0}^{j},x_{c}]=E[x_k|x_j,x_{c}]
\end{align} 
for every $j,k \in [0,N], j<k$. 

\end{lemma}
\begin{proof}
Necessity: By Lemma \ref{CML_Definition_Expectation}, for a $CM_c$ sequence $[x_k]$, $\eqref{GaussianCML_E_1}$ holds.

Sufficiency: Let $[x_k]$ be a Gaussian sequence for which $\eqref{GaussianCML_E_1}$ holds. The conditional covariance can be calculated as
\begin{align*}
&\text{Cov}(x_k|[x_{i}]_{0}^{j},x_{c})=E\Big[\Big(x_k- E[x_k|[x_{i}]_{0}^{j},x_{c}]\Big)\Big(\cdot\Big)'\Big | [x_{i}]_{0}^{j},x_{c}\Big]
\end{align*}
On the other hand, for conditional expectation we have 
\begin{align*}
E[&(x_k-E[x_k|[x_{i}]_{0}^{j},x_{c}])g([x_{i}]_{0}^{j},x_{c})]=0
\end{align*}
for every Borel measurable function $g$. Thus, $x_k-E[x_k|[x_{i}]_{0}^{j},x_{c}]$ is orthogonal to (and due to Gaussianity independent of) $[x_{i}]_{0}^{j}$ and $x_{c}$. Therefore, noting $\eqref{GaussianCML_E_1}$, we have 
\begin{align}
\text{Cov}(x_k|[x_{i}]_{0}^{j},x_{c})&=E\Big[\Big(x_k- E[x_k|x_j,x_{c}]\Big)\Big(\cdot\Big)'\Big]\nonumber \\
&=E\Big[\Big(x_k- E[x_k|x_j,x_{c}]\Big)\Big(\cdot\Big)'|x_j,x_c\Big]\nonumber\\
&=\text{Cov}(x_k|x_j,x_{c})\label{GaussianCML_Cov}
\end{align}

Due to Gaussianity, $\eqref{GaussianCML_E_1}$ and $\eqref{GaussianCML_Cov}$ lead to the equality of the corresponding conditional density. In other words, the Gaussian conditional density is completely determined by its conditional expectation \cite{Doob}. Therefore, $\eqref{CML_Expectation_1}$ holds and the sequence $[x_k]$ is $CM_c$.
\end{proof}

\section{Dynamic Models of $CM_c$ Sequences}\label{Section_CML_Dynamic}

\subsection{Forward Model}\label{CMc_Forward}

A dynamic model for the ZMNG reciprocal sequence was presented in \ [33]. Inspired by it, a model for evolution of the ZMNG $CM_c$ sequence, called a $CM_c$ model, is presented next. Lemma \ref{CML_Dynamic_Forward_Necessity} demonstrates construction of a $CM_c$ model.

\begin{lemma}\label{CML_Dynamic_Forward_Necessity}
Let $[x_k]$ be a ZMNG $CM_c$ sequence with covariance function $C_{k_1,k_2}$. Then, it is governed by
\begin{align}
x_k&=G_{k,k-1}x_{k-1}+G_{k,c}x_c+e_k, \quad k \in [1,N] \setminus \lbrace c \rbrace
\label{CML_Dynamic_Forward}\\
x_c&=e_c, \quad x_0=G_{0,c}x_c+e_0 \, \, (\text{for} \, \, c=N) \label{CML_Forward_BC2}
\end{align}
where $[e_k]$ is a zero-mean white NG sequence with covariances $G_k$.

\end{lemma}
\begin{proof}
We prove the following: (i) model construction, (ii) boundary conditions and the whiteness of $[e_k]$. Nonsingularity of $G_k, k \in [0,N]$ can be easily proved (we skip it).

(i) Model construction: 

Since $[x_k]$ is $CM_c$, by Lemma \ref{GaussianCML_Definition_Expectation} for every $k \in [1,N] \setminus \lbrace c \rbrace$ we have
\begin{align}
E[x_k| [x_i]_0^{k-1},x_c]&=E[x_k|x_{k-1},x_c]\label{CML_p}
\end{align}
Since $[x_k]$ is Gaussian, for $c=0$ and $k=1$ we have $E[x_k|x_{k-1},x_c]=C_{1,0}C_0^{-1}x_0$. Let $G_{1,0} \triangleq \frac{1}{2} C_{1,0}C_0^{-1}$. For other $c$ and $k$ values (i.e., $c=0$ and $k \in [2,N]$, and $c=N$ and $k \in [1,N-1]$),
 \begin{align*}
E[x_k|x_{k-1},x_c]=[C_{k,k-1} \quad C_{k,c}] \left[
\begin{array}{cc}
C_{k-1} & C_{k-1,c}\\
C_{c,k-1} & C_{c}
\end{array} \right]^{-1}\left[\begin{array}{c}
x_{k-1}\\
x_c
\end{array}\right]
\end{align*}
Let
\begin{align*}
&[G_{k,k-1} \quad G_{k,c}] \triangleq [C_{k,k-1} \quad C_{k,c}] \left[
\begin{array}{cc}
C_{k-1} & C_{k-1,c}\\
C_{c,k-1} & C_{c}
\end{array} \right]^{-1}
\end{align*}
So, for every $k \in [1,N] \setminus \lbrace c \rbrace$ and $c \in \lbrace 0,N \rbrace$,
\begin{align}
E[x_k|x_{k-1},x_c]=G_{k,k-1}x_{k-1}+G_{k,c}x_c\label{CML_p2}
\end{align}
Define $e_k$, $k \in [1,N] \setminus \lbrace c \rbrace$, as
\begin{align}\label{eLk}
e_k&=x_k-E[x_k|x_{k-1},x_c]\\
&=x_k-G_{k,k-1}x_{k-1}-G_{k,c}x_c\nonumber
\end{align}
Then, for $c=0$ and $k=1$, $G_1 \triangleq \text{Cov}(e_1)=C_1-C_{1,0} C_0^{-1}C_{1,0}'$. For other $c$ and $k$ values,
\begin{align*}
G_k \triangleq \text{Cov}(e_k)= C_{k}-[C_{k,k-1} \quad C_{k,c}]\left[
\begin{array}{cc}
C_{k-1} & C_{k-1,c}\\
C_{c,k-1} & C_{c}
\end{array} \right]^{-1}
[C_{k,k-1} \quad C_{k,c}]'
\end{align*}
$[e_k]_{[1,N]\setminus \lbrace c \rbrace}$ is a zero-mean white Gaussian sequence uncorrelated with $x_0$ and $x_{c}$. It can be verified as follows. By the definition of conditional expectation and based on $\eqref{CML_p}$ we have
\begin{align}\label{E_CML_e}
E[(x_k-E[x_k|&x_{k-1},x_c])g([x_j]_0^{k-1},x_c)]=\nonumber\\&E[(x_k-E[x_k|[x_i]_{0}^{k-1},x_c])g([x_j]_0^{k-1},x_c)]=0
\end{align}
for every Borel measurable function $g$. Thus, by $\eqref{eLk}$ and $\eqref{E_CML_e}$, $e_k$ is uncorrelated with $[x_i]_0^{k-1}$ and $x_c$. Then, for $k \geq j$,
\begin{align}
E[e_ke_j']&=E[e_k(x_j-G_{j,j-1}x_{j-1}-G_{j,c}x_c)']=\left\{
\begin{array}{cc}
G_k & k=j\\
0 & \text{otherwise}
\end{array}\right.
\end{align}
Likewise for $j \geq k$. Therefore, we have 
\begin{align*}
E[e_ke_j']=\left\{
\begin{array}{cc}
G_{k} & k=j\\
0 & k \neq j
\end{array}\right.
\end{align*}
So, $[e_k]_{[1,N] \setminus \lbrace c \rbrace}$ is white.

(ii) Boundary conditions: 

For $c=0$, we have $G_0\triangleq C_0$. Let $c=N$. 
Since $x_0$ and $x_N$ are jointly Gaussian, we have $E[x_0|x_N]=G_{0,N}x_N$, where $G_{0,N}=C_{0,N}C_N^{-1}$. Then, we define $e_0 \triangleq x_0-G_{0,N}x_N$, where $e_0$ is a ZMNG vector with covariance $G_0=C_{0}-C_{0,N}C_{N}^{-1}C_{0,N}'$. Also, by the definition of conditional expectation, $e_0$ is uncorrelated with $x_N$ (because $E[(x_0-E[x_0|x_N])g(x_N)]=0$ for every Borel measurable function $g$). Also, for notational unification $e_N \triangleq x_N$ with covariance $G_N \triangleq C_N$. By $\eqref{E_CML_e}$, $[e_k]$ is white.  
\end{proof}

It is important that a dynamic model gives a unique covariance function of the corresponding sequence \cite{Levy_Dynamic}. This is the case for model $\eqref{CML_Dynamic_Forward}$--$\eqref{CML_Forward_BC2}$.

\begin{lemma}\label{CML_Forward_Well_Definedness}
\label{well-definedness}
Model $\eqref{CML_Dynamic_Forward}$--$\eqref{CML_Forward_BC2}$ for every parameter value admits a unique covariance function.

\end{lemma}

\begin{lemma}\label{CML_Forward_Nonsingularity}
$[x_k]$ governed by $\eqref{CML_Dynamic_Forward}$--$\eqref{CML_Forward_BC2}$ is always nonsingular (for every parameter value).

\end{lemma}

By the above lemmas, a model for the ZMNG $CM_c$ sequence was constructed and some related properties were studied. Now, we present the main result for the $CM_c$ model.

\begin{theorem}\label{CML_Dynamic_Forward_Proposition}
A ZMNG sequence $[x_k]$ with covariance function $C_{k_1,k_2}$ is $CM_c$ iff it obeys $\eqref{CML_Dynamic_Forward}$--$\eqref{CML_Forward_BC2}$.
\end{theorem}
\begin{proof}
Theorem \ref{CML_Dynamic_Forward_Proposition} is proved based on Lemma \ref{GaussianCML_Definition_Expectation}. The necessity was proved in Lemma \ref{CML_Dynamic_Forward_Necessity}. So, we just need to prove the sufficiency. This amounts to prove $[x_k]$ is (i) nonsingular and (ii) Gaussian $CM_c$. Lemma \ref{CML_Forward_Nonsingularity} has established (i). So, we just need to prove (ii). Since $[x_k]$ is Gaussian, by Lemma \ref{GaussianCML_Definition_Expectation} $[x_k]$ is $CM_c$ if $E[x_k|[x_i]_0^j,x_c]=E[x_k|x_j,x_c]$ for every $j,k \in [0,N] \setminus \lbrace c \rbrace, j < k$. From $\eqref{CML_Dynamic_Forward}$ we have $x_k=G_{k,j}x_j+G_{k,c|j}x_c+e_{k|j}$, where the matrices $G_{k,j}$ and $G_{k,c|j}$ can be obtained from parameters of $\eqref{CML_Dynamic_Forward}$, and $e_{k|j}$ is a linear combination of $[e_l]_{j+1}^k$. Since $[e_k]$ is white, $[e_l]_{j+1}^k$ (and so $e_{k|j}$) is uncorrelated with $[x_k]_0^j$ and $x_c$. Thus, we have $E[x_k|[x_i]_0^j,x_c]=E[x_k|x_j,x_c]$, meaning that $[x_k]$ is $CM_c$.
\end{proof}

\section{Characterization of $CM_c$ Sequences}\label{Section_CML_Characterization}

\begin{definition}\label{CMc_Matrix}
A symmetric positive definite matrix is called $CM_L$ if it has form $\eqref{CML}$ and $CM_F$ if it has form $\eqref{CMF}$:
\begin{align}
&\left[
\begin{array}{ccccccc}
A_0 & B_0 & 0 & \cdots & 0 & 0 & D_0\\
B_0' & A_1 & B_1 & 0 & \cdots & 0 & D_1\\
0 & B_1' & A_2 & B_2 & \cdots & 0 & D_2\\
\vdots & \vdots & \vdots & \vdots & \vdots & \vdots & \vdots\\
0 & \cdots & 0 & B_{N-3}' & A_{N-2}  & B_{N-2} & D_{N-2}\\
0 & \cdots & 0 & 0 & B_{N-2}' & A_{N-1} & B_{N-1}\\
D_0' & D_1' & D_2' & \cdots & D_{N-2}' & B_{N-1}' & A_N
\end{array}\right]\label{CML}\\
&\left[
\begin{array}{ccccccc}
A_0 & B_0 & D_2 & \cdots & D_{N-2} & D_{N-1} & D_{N}\\
B_0' & A_1 & B_1 & 0 & \cdots & 0 & 0\\
D_2' & B_1' & A_2 & B_2 & \cdots & 0 & 0\\
\vdots & \vdots & \vdots & \vdots & \vdots & \vdots & \vdots\\
D_{N-2}' & \cdots & 0 & B_{N-3}' & A_{N-2}  & B_{N-2} & 0\\
D_{N-1}' & \cdots & 0 & 0 & B_{N-2}' & A_{N-1} & B_{N-1}\\
D_{N}' & 0 & 0 & \cdots & 0 & B_{N-1}' & A_N
\end{array}\right]\label{CMF}
\end{align}

\end{definition}

To refer to both $CM_L$ and $CM_F$ matrices we call them $CM_c$. A $CM_c$ matrix for $c=N$ is $CM_L$ and for $c=0$ is $CM_F$. The following theorem presents a characterization of the NG $CM_c$ sequence. It can be proved based on the $CM_c$ dynamic model of Theorem \ref{CML_Dynamic_Forward_Proposition}. 

\begin{theorem}\label{CML_Characterization}
A NG sequence with covariance matrix $C$ is $CM_c$ iff $C^{-1}$ has the $CM_c$ form.

\end{theorem}

Theorem \ref{CML_Characterization} can be also verified based on the relationship between the covariance matrix and conditional independence between some Gaussian variables \cite{Graphical}.

A characterization of the NG reciprocal sequence is as follows \cite{Levy_Dynamic}.

\begin{theorem}\label{Reciprocal_Char}
A NG sequence with covariance matrix $C$ is reciprocal iff $C^{-1}$ is cyclic tri-diagonal (i.e., $\eqref{CML}$ with $D_1=\cdots=D_{N-2}=0$). 

\end{theorem}

A characterization of the NG Markov sequence is as follows \cite{Ackner}.

\begin{theorem}\label{Markov_Characterization}
A NG sequence with covariance matrix $C$ is Markov iff $C^{-1}$ is tri-diagonal (i.e., $\eqref{CML}$ with $D_0=\cdots=D_{N-2}=0$).

\end{theorem}

Markov sequences are reciprocal, and reciprocal sequences are $CM_c$ \cite{CM_Part_II_A_Conf}.

\section{Conclusions and Applications}\label{Section_Summary}

Conditioning is a very powerful tool in probability theory. The Markov property, defined based on conditioning, is very important and widely used in application. The conditionally Markov (CM) process, which combines conditioning and the Markov property, is a general class of stochastic processes, including the Markov process as a special case. Different ways of combination lead to different classes of CM processes. Therefore, a systematic approach and a wide variety of choices are provided for modeling random phenomena/problems. 

We have elaborated general definitions of CM sequences, studied and characterized nonsingular Gaussian (NG) CM sequences, and obtained their dynamic models. The $CM_c$ sequence is an important class of CM sequences. Markov and reciprocal sequences are special $CM_c$ sequences. That is why characterizations of NG Markov and reciprocal sequences are special cases of those of $CM_c$ sequences. Therefore, the results of this paper build a foundation for studying reciprocal sequences from the CM viewpoint. This viewpoint leads to simple and desirable results for reciprocal sequences. For example, the existing model for NG reciprocal sequences \cite{Levy_Dynamic} is driven by \textit{colored} noise. However, it is possible to obtain \textit{reciprocal} $CM_c$ models driven by \textit{white} noise governing NG reciprocal sequences \cite{CM_Part_II_A_Conf}, \cite{CM_Explicitly}--\cite{CM_Journal_Algebraically}. Due to whiteness of the dynamic noise, these models are simple. In addition, viewing the reciprocal sequence as a special CM sequence gives more insight into the reciprocal sequence and reveals new properties of it.  

The main components of motion trajectories without destination are: an origin and an evolution law. Markov sequences can be used for modeling such trajectories based on their initial density and evolution law. However, Markov sequences are not flexible enough for modeling the main components of motion trajectories with destination information (i.e., an origin, evolution, and a destination). This is because the density of a Markov sequence at the destination is determined by its initial density and the evolution law. The $CM_L$ sequence, as a more general class of stochastic sequences, can model trajectories with destination information. The main components of such trajectories can be seen in the $CM_L$ dynamic model. $x_N$ models the state of the destination. Conditioned on $x_N$, the evolution law is Markov, which is simple and desired for application. Also, $x_0$ models the state at the origin. In addition, the $CM_L$ sequence can have any relationship between the states at the origin and at the destination. Moreover, due to whiteness of the dynamic noise, estimation of a sequence governed by a $CM_L$ model is straighforward. This is particularly useful for trajectory prediction, which is a critical task in air traffic control. In \cite{DD_Conf}, a $CM_L$ model was used for trajectory modeling with destination information. Also, a CM sequence was proposed in \cite{DW_Conf} for trajectory modeling with waypoint information.

Singular/nonsingular Gaussian CM and reciprocal sequences and their application were studied in \cite{CM_SingularNonsingular}--\cite{Thesis_Reza}.

\subsubsection*{Acknowledgments}

Research was supported by NASA Phase03-06 through grant NNX13AD29A.

\end{document}